\documentclass[12pt]{article}
\usepackage{amssymb,amsmath,amsthm,amsfonts,enumerate, bbm, mathdots}

\usepackage{latexsym}
\usepackage{amscd}
\usepackage{stmaryrd}
\usepackage{color}
\usepackage[all]{xypic}
\usepackage{epsfig}
\usepackage{graphics}
\usepackage{ifthen}
\usepackage{varioref}
\usepackage{rotating}
\usepackage{extarrows}
\usepackage{cite}
\usepackage{mathrsfs}

\numberwithin{equation}{section}

\theoremstyle{plain}
\newtheorem{thm}{Theorem}[section]

\newtheorem{lem}[thm]{Lemma}

\newtheorem{rem}[thm]{Remark}

\newmuskip\pFqmuskip
\newcommand*\pFq[6][8]{%
 \begingroup % only local assignments
 \pFqmuskip=#1mu\relax
 \mathchardef\normalcomma=\mathcode`,
 \mathcode`\,=\string"8000
 \begingroup\lccode`\~=`\,
 \lowercase{\endgroup\let~}\pFqcomma
 {}_{#2}F_{#3}{\left(\genfrac..{0pt}{}{#4}{#5};#6\right)}%
 \endgroup
}
\newcommand{\pFqcomma}{{\normalcomma}\mskip\pFqmuskip}

\makeatletter
\newenvironment{proofof}[1]{\par
 \pushQED{\qed}%
 \normalfont \topsep6\p@\@plus6\p@\relax
 \trivlist
 \item[\hskip\labelsep
    \bfseries
  Proof of #1\@addpunct{.}]\ignorespaces
}{%
 \popQED\endtrivlist\@endpefalse
}
\makeatother

\definecolor{dg}{rgb}{0.0625,0.64,0.0625}
\usepackage[%dvipdfm,
      pdfstartview=FitH,
      CJKbookmarks=true,
      bookmarksnumbered=true,
      bookmarksopen=true,
      colorlinks,
      pdfborder=001,
      linkcolor=blue,
      anchorcolor=green,
      citecolor=red
      ]{hyperref}

\definecolor{deepmaroon}{RGB}{100,0,20}   % 深暗红
\definecolor{deepnavy}{RGB}{0,30,80}    % 藏青深蓝
\definecolor{deepforest}{RGB}{0,60,30}   % 森林深绿
\definecolor{deeppurple}{RGB}{128,0,255}   % 暗紫
\definecolor{deepteal}{RGB}{0,70,70}    % 深湖蓝
\definecolor{deepolive}{RGB}{60,65,30}   % 暗橄榄绿
\definecolor{darkredheavy}{RGB}{120,0,0}   % 浓深红
\definecolor{lightgreen}{RGB}{0,255,255}  % 纯浓藏青
\definecolor{darkgreenheavy}{RGB}{0,80,0}  % 浓墨绿
\definecolor{darkpurpleheavy}{RGB}{90,0,120} % 浓深紫
\definecolor{darktealheavy}{RGB}{0,90,90}  % 浓深青
\definecolor{darkorangeheavy}{RGB}{160,60,0} % 深橘棕
\definecolor{matcha}{RGB}{130,190,140}
\definecolor{emeraldDeep}{RGB}{0,115,90}
\definecolor{grassBright}{RGB}{20,180,60}

\newfont{\scyr}{wncyr10 scaled 550}

\def\proof{\noindent {\bf Proof.\;}}

\allowdisplaybreaks

\newcommand{\A}{\mathcal{A}}
\newcommand{\F}{\mathbb{F}}

\newcommand{\Z}{\mathbb{Z}}
\newcommand{\card}{\#}

\begin{document}

\title{A proof of a weighted sum conjecture for finite multiple zeta values of level two}
\date{~}
\author{{Zhonghua Li${}^{a,}$\thanks{Email: zhonghua\_li@tongji.edu.cn},\quad{Zhenlu Wang${}^{b,}$\thanks{Corresponding author. Email: zhenluwang@tzc.edu.cn}\quad and\quad{Lihui Zhang${}^{a,}$}\thanks{Email: 2433952@tongji.edu.cn}}}\\[1mm]\small a. School of Mathematical Sciences,\\ \small Key Laboratory of Intelligent Computing and Applications (Tongji University), \\ \small Ministry of Education, \small Tongji University, Shanghai 200092, China\\
\small b. School of Artificial Intelligence, \\ \small Taizhou University, Taizhou 318000, Zhejiang, China}

\maketitle

\begin{abstract}
M. Kaneko, T. Murakami and A. Yoshihara introduced finite multiple zeta values of level two and conjectured a weighted sum formula for indices whose components belong to $\{1,2\}$. In this paper, we use generating functions and linear recurrence relations to prove the conjecture. More precisely, we reduce the problem to a polynomial identity and solve the resulting second-order difference equation by identifying its specialized solutions with $_4F_3$ hypergeometric polynomials of Racah-type. 
\end{abstract}

{\small

}

%%---------------------------------------------------------------------------
%%------------------------Content-------------------------------------------
%%----------------------------------------------------------------------------

\section{Introduction}
A tuple of positive integers $\boldsymbol{k}=(k_1,\ldots,k_r)$ is called an index. For an index $\boldsymbol{k}=(k_1,\ldots,k_r)$, the multiple
zeta value (MZV) is defined by
\begin{align*}
 \zeta(\boldsymbol{k})
 =\sum_{1\leq m_1<\cdots<m_r}
  \frac{1}{m_1^{k_1}\cdots m_r^{k_r}},
\end{align*}
with the usual convergence condition $k_r\geq 2$. These values were independently and systematically studied by Hoffman \cite{Hoffman1992} and Zagier \cite{Zagier1994} in the early 1990s and have since played an important role in both mathematics and theoretical physics. Multiple zeta values satisfy rich algebraic relations arising from their series and iterated-integral representations, including the shuffle and harmonic products, regularized double-shuffle relations, and derivation relations \cite{Hoffman1997,IharaKanekoZagier2006}. Among the classical relations for multiple zeta values are the sum formula and its weighted generalizations \cite{Granville1997,OhnoZudilin2008,GuoXie2009}. Generating functions provide an effective tool for studying such identities, as illustrated by the Ohno--Zagier formula \cite{OhnoZagier2001}.

The modular aspects of multiple zeta values were independently studied by Hoffman \cite{Hoffman2015} and Zhao \cite{Zhao2008} through finite multiple harmonic sums. This led naturally to the study of finite multiple zeta values in the adèle-like ring $\mathcal{A}$ introduced by Kontsevich and subsequently used by Kaneko and Zagier \cite{KanekoZagier2000}. Let $\mathcal P$ be the set of primes and set
\begin{align*}
 \A=\prod_{p\in\mathcal P}\Z/p\Z\Big/\bigoplus_{p\in\mathcal P}\Z/p\Z.
\end{align*}
Then, for an index $\boldsymbol{k}=(k_1,\ldots,k_r)$, the finite multiple zeta value (FMZV) is defined by
\begin{align*}
 \zeta_{\A}(\boldsymbol{k})
 =\left(
 \sum_{1\leq m_1<\cdots<m_r<p}
 \frac{1}{m_1^{k_1}\cdots m_r^{k_r}}\pmod p
 \right)_{p}\in\A.
\end{align*}

The framework of Kaneko and Zagier reveals striking parallels between finite and classical multiple zeta values \cite{Kaneko2019,Zhao2016}. This connection is further reflected in the Kaneko--Zagier conjecture, which relates finite multiple zeta values to classical multiple zeta values modulo $\zeta(2)$. Although the conjecture remains open, many parallel relations are known on both sides, including finite analogues of sum and weighted sum formulas \cite{SaitoWakabayashi2015,HiroseMuraharaSaito2019}.

Recently, level-two variants arise by imposing parity restrictions on summation variables. Important classical examples include Hoffman's multiple $t$-values \cite{Hoffman2019} and Kaneko--Tsumura's multiple $T$-values \cite{KanekoTsumura2020}. For an index $\boldsymbol{k}=(k_1,\ldots,k_r)$, Kaneko, Murakami and Yoshihara \cite{KanekoMurakamiYoshihara2023} introduced the finite multiple zeta values of level two (level-two FMZVs) by
\begin{align*}
   \zeta_{\A}^{(2)}(\boldsymbol{k})
 =\left(
 \sum_{1\leq m_1<\cdots<m_r<p/2}
 \frac{1}{m_1^{k_1}\cdots m_r^{k_r}}\pmod p
 \right)_p\in\A.
\end{align*}
Unlike ordinary FMZVs, the summation here is truncated at $p/2$ rather than at $p$. In particular, the level-two FMZVs may also be viewed as finite analogues of multiple $t$-values, up to a constant multiple. Their paper establishes parity results and several sum
formulas and ends with the weighted sum conjecture that is the subject of this paper. More recently, finite analogues of multiple $T$-values were studied by Zhao \cite{ZhaoAxioms2024,ZhaoFoundations2024}, and level-two FMZVs were further generalized within the framework of finite multiple mixed values, with arbitrary parity restrictions on the summation\cite{LiWang2026,ZhaoFoundations2024}.

Let $p$ be an odd prime and set $N=(p-1)/2$. Let $\F_p=\mathbb{Z}/p\mathbb{Z}$ denote the finite field with $p$ elements. Our main result is the following theorem.
\begin{thm}\label{thm:main-poly}
In the polynomial ring $\F_p[x,z]$, we have 
\begin{align}\label{eq:main-poly}
\sum_{r=0}^{N}\sum_{1\le m_1<\cdots<m_r\le N}\prod_{j=1}^{r}\left(\frac{x}{m_j}+\frac{2(-1)^jz}{m_j^2}\right)=\prod_{m=1}^N\left(1+\frac{x}{m}+\frac{z}{m^2}\right),
\end{align}
where the inner sum for $r=0$ consists of the empty tuple and the corresponding empty product is understood to be $1$.
\end{thm}

The proof of Theorem \ref{thm:main-poly} is based on a recurrence relation related to the polynomial on the left-hand side of \eqref{eq:main-poly}. By transforming the recurrence relation into a second-order difference equation of hypergeometric type, we obtain a suitable hypergeometric representation and determine the roots of the polynomial, from which the desired factorization follows.

As a consequence of Theorem \ref{thm:main-poly}, we prove the conjecture of Kaneko--Murakami--Yoshihara \cite[Conjecture~3.8]{KanekoMurakamiYoshihara2023}.
\begin{thm}\label{thm:KMY}
For integers $r,a$ with $r\ge1$ and $0\le a\le r$, we have
\begin{align}\label{eq:KMY-conj}
 \sum_{\substack{k_i\in\{1,2\}\\\#\{i:k_i=2\}=a}}
 \left(2^a(-1)^{\#\{i:\ i\text{ odd},\ k_i=2\}}-1
 \right)\zeta_{\A}^{(2)}(k_1,\ldots,k_r)=0.
\end{align}
\end{thm}
Equation \eqref{eq:KMY-conj} can be regarded as a weighted sum formula for level-two FMZVs, whose weights depend on both the number of entries equal to $2$ in the index and the parity of their positions.

In the next section, we first establish several preliminary lemmas needed for the proofs of Theorems \ref{thm:main-poly} and \ref{thm:KMY}, and then prove the two main theorems using these preparatory results.

%%---------------------------%%--------------------
\section{Generating functions and the proofs}
Recall that $p$ is an odd prime and $N=(p-1)/2$. For integers $n$ with $0\le n\le N$, we define
\begin{align*}
G_n=G_n(x,z)=\sum_{r=0}^{n}\sum_{1\le m_1<\cdots<m_r\le n}\prod_{j=1}^{r}\left(\frac{x}{m_j}+\frac{2(-1)^jz}{m_j^2}\right),\label{eq:Fdef}
\end{align*}
where the term with $r=0$ is $1$.
Splitting $G_n$ according to the parity of $r$, we write $G_n=E_n+O_n$, where 
\begin{align*}
E_n=E_n(x,z)=\sum_{\substack{r=0\\r:even}}^{n}\sum_{1\le m_1<\cdots<m_r\le n}\prod_{j=1}^{r}\left(\frac{x}{m_j}+\frac{2(-1)^jz}{m_j^2}\right)
\end{align*}
and 
\begin{align*}
O_n=O_n(x,z)=\sum_{\substack{r=0\\r:odd}}^{n}\sum_{1\le m_1<\cdots<m_r\le n}\prod_{j=1}^{r}\left(\frac{x}{m_j}+\frac{2(-1)^jz}{m_j^2}\right).
\end{align*}
Since $E_0=1$ and $O_0=0$, we have $G_0=1$. It is easy tho find that 
\begin{equation}
\begin{pmatrix}E_n\\ O_n\end{pmatrix}
=\begin{pmatrix}
1 & \frac{x}{n}+\frac{2z}{n^2}\\[1mm]
\frac{x}{n}-\frac{2z}{n^2} & 1
\end{pmatrix}
\begin{pmatrix}E_{n-1}\\ O_{n-1}\end{pmatrix}.\label{eq:EO-matrix}
\end{equation}

Define 
\begin{align*}
  g_n=g_n(x,z)=(n!)^2G_n(x,z).
\end{align*}

We first establish the following recurrence relation for $g_n$.
\begin{lem}\label{lem:g-rec}
  For $1\le n\le N$, we have
\begin{align}
  g_{n+1}
 =(2n^2+2n+1+x)g_n
 -(n^4-n^2x^2+4z^2)g_{n-1}\label{eq:g-recurrence}
\end{align}
with initial values $g_0=1$ and $g_1=1+x-2z$. Moreover, $g_N$ is a monic polynomial of degree $N$ in $z$ over $\F_p(x)$.
\end{lem}
\proof
The initial values $g_0=1$ and $g_1=1+x-2z$ follow directly from the definition. Set $h_n=(n!)^2(E_n-O_n)$. Since $g_n=(n!)^2(E_n+O_n)$, we have
\begin{equation*}
\begin{pmatrix}g_n\\ h_n\end{pmatrix}
=(n!)^2\begin{pmatrix}
1 & 1\\[1mm]
1 & -1
\end{pmatrix}
\begin{pmatrix}E_{n}\\ O_{n}\end{pmatrix}.
\end{equation*}
Using \eqref{eq:EO-matrix}, we obtain
\begin{equation}\label{eq:gh-recurrence}
\begin{pmatrix}g_n\\ h_n\end{pmatrix}
=\begin{pmatrix}
n(n+x) & -2z\\[1mm]
2z & n(n-x)
\end{pmatrix}
\begin{pmatrix}g_{n-1}\\ h_{n-1}\end{pmatrix}.
\end{equation}
Taking the first component gives
\begin{align}
g_n=n(n+x)g_{n-1}-2zh_{n-1}.\label{eq:g_n}
\end{align}
Replacing $n$ by $n+1$ in \eqref{eq:gh-recurrence} and then applying \eqref{eq:gh-recurrence} once more, we obtain
\begin{align}
  g_{n+1}&=\left(n(n+1)(n+x)(n+1+x)-4z^2\right)g_{n-1}\notag\\&\qquad-2z\left((n+1)(n+1+x)+n(n-x)\right)h_{n-1}.\label{eq:g_{n+1}}
\end{align}
Eliminating $h_{n-1}$ from \eqref{eq:g_n} and \eqref{eq:g_{n+1}}, we get the recurrence relation \eqref{eq:g-recurrence}.

It remains to prove the last assertion. The recurrence relation \eqref{eq:g-recurrence}, together with $g_0=1$ and $g_1=1+x-2z$, shows inductively that $\deg_z g_n=n$ and the coefficient of $z^n$ in $g_n$ is
\begin{align*}
c_n=(-1)^{n(n+1)/2}2^n.
\end{align*}
For $N=(p-1)/2$, Euler's criterion gives
\begin{align*}
  2^N\equiv (-1)^{(p^2-1)/8}=(-1)^{N(N+1)/2}\pmod p.
\end{align*}
Therefore, we have
\begin{align*}
c_N=(-1)^{N(N+1)/2}2^N\equiv 1\pmod p.
\end{align*}
Hence $g_N$ is a monic polynomial of degree $N$ in $z$ over $\F_p(x)$, as desired.
\qed

We next show that the polynomial $g_N$ in $z$ over $\F_p(x)$ vanishes at
\begin{align*}
  -s(s+x),\quad 1\le s\le N.
\end{align*}
More precisely, we have the following theorem.
\begin{thm}\label{thm:g-roots}
For each integer $s \in\{1,2,\ldots,N\}$, we have
\begin{align*}
g_N(x,-s(s+x))=0.
\end{align*}
\end{thm}

The proof of Theorem \ref{thm:g-roots} requires several preparatory steps. We therefore begin by establishing two lemmas that will be used in the proof. The main idea is to transform the recurrence relation \eqref{eq:g-recurrence} into a second-order difference equation of hypergeometric type and identify its distinguished polynomial solution. This solution admits a representation in terms of a generalized hypergeometric function ${}_4F_3$ evaluated at $1$, which is closely related to the standard hypergeometric form of the Racah-type polynomials.

For fixed $s\in\{1,\ldots,N\}$, define $\rho$ and $\sigma$ by
\begin{align*}
\rho+\sigma=2s+x,
\qquad
\rho\sigma=2s(s+x).%\label{eq:rhosigma}
\end{align*}
Set
\begin{align*}
L_n&=(n-\rho)(n-\sigma)
=n^2-(2s+x)n+2s(s+x),%\label{eq:Ln}
\\
R_n&=(n+\rho)(n+\sigma)
=n^2+(2s+x)n+2s(s+x).%\label{eq:Rn}
\end{align*}
Then
\begin{align*}
L_nR_n=n^4-n^2x^2+4s^2(s+x)^2.%\label{eq:LR}
\end{align*}

Next, define
\begin{align*}
d_n=\prod_{j=1}^nL_j=(1-\rho)_n(1-\sigma)_n,%\label{eq:d}
\end{align*}
where the Pochhammer symbol $(a)_n$ is defined by
$$(a)_n=\begin{cases}
1 & \text{if}\quad n=0, \\
a(a+1)\cdots(a+n-1) & \text{if}\quad n>0.
\end{cases}$$
Since $L_j\neq 0$ in $\F_p(x)$ for $(1\le j\le N)$, we may define
\begin{align*}
u_n=u_n(x)=\frac{g_n(x,-s(s+x))}{d_n}.%\label{eq:u-norm}
\end{align*}
After specializing to $z=-s(s+x)$, the recurrence relation
\eqref{eq:g-recurrence} becomes
\begin{align}
L_{n+1}u_{n+1}-A_nu_n+R_nu_{n-1}=0,\label{eq:urec}
\end{align}
where $A_n=2n^2+2n+1+x$.

Now define the second-order difference operator $\mathcal{D}_s$ by
\begin{align}
(\mathcal{D}_sf)(n)=L_{n+1}\bigl(f(n+1)-f(n)\bigr)-R_n\bigl(f(n)-f(n-1)\bigr).\label{eq:D}
\end{align}
A direct calculation gives
\begin{align*}
L_{n+1}+R_n-A_n=2(2s-1)(s+x).%\label{eq:kappa}
\end{align*}
Set
\begin{align*}
\kappa_s=2(2s-1)(s+x).
\end{align*}
Then the recurrence relation \eqref{eq:urec} for $u_n$  can be written as
\begin{align}
  \mathcal{D}_su=-\kappa_su.\label{eq:Drec}
\end{align}

\begin{rem}
The normalization that transforms \eqref{eq:g-recurrence} into \eqref{eq:Drec} may be viewed as a multiplicative similarity transformation of the underlying second-order difference operator. Such transformations are standard in the theory of difference equations and discrete orthogonal polynomials, see \cite{OdakeSasaki2008,Odake2017}. 
\end{rem}

Define
\begin{align}
\phi_k(n)=(2n+1)(-1)^k(-n)_k(n+1)_k\quad\text{for}\quad k\ge0.
\label{eq:phi}
\end{align}
It is easy to find that 
 \begin{align*}
  \phi_k(-n-1)=-\phi_k(n),\quad \phi_0(0)=1 \quad\text{and}\quad\phi_k(0)=0\quad(k\ge1).
 \end{align*}
Moreover, the following identities hold:
\begin{align}
&(2n+1)(n-k+1)\phi_k(n+1)=(2n+3)(n+k+1)\phi_k(n),\label{eq:phi-forward-ratio}\\
&(2n+1)(n+k)\phi_k(n-1)=(2n-1)(n-k)\phi_k(n),\label{eq:phi-backward-ratio}
\end{align}
and
\begin{align}
&\phi_k(n)=(n-k+1)(n+k)\phi_{k-1}(n).\label{eq:phi-lower}
\end{align}

\begin{rem}
Consider the quadratic lattice $\Lambda(n)=n(n+1)$. Then we can rewrite $\phi_k(n)$ as $\phi_k(n)=(2n+1)\prod_{j=0}^{k-1}\left(\Lambda(n)-\Lambda(j)\right)$. Thus, up to the prefactor (2n+1), the family $\{\phi_k(n)\}_{k\ge0}$ forms the natural Newton-type basis associated with the quadratic lattice $\Lambda(n)=n(n+1)$. Such quadratic lattices arise naturally in the theory of Racah and dual Hahn polynomials, see \cite{KLS2010}. In particular, this choice of basis is compatible with the Racah-type polynomial structure that appears below.
\end{rem}

\begin{lem}\label{lem:tri-action}
  For $k\ge0$, we have
\begin{align*}
\mathcal D_s\phi_k
=\lambda_k\phi_k+\mu_k\phi_{k-1},
%\label{eq:triangular}
\end{align*}
where we adopt the convention $\phi_{-1}=0$, and
\begin{align*}
\lambda_k
&=-2(2k+1)(x+2s-k-1),\quad\mu_k=2k(2k+1)(k-\rho)(k-\sigma).
\end{align*}
\end{lem}
\proof
Using \eqref{eq:D}, \eqref{eq:phi-forward-ratio} and \eqref{eq:phi-backward-ratio}, we obtain
\begin{align*}
\mathcal D_s\phi_k=\frac{2(2k+1)}{(2n+1)(n-k+1)(n+k)}A\,\phi_k,
%\label{eq:Dphi-first}
\end{align*}
where
\begin{align*}
A=(n+1)(n+k)L_{n+1}-n(n-k+1)R_n.
\end{align*}
Recall that
\begin{align*}
L_n=n^2-(2s+x)n+2s(s+x)
\quad\text{and}\quad
R_n=n^2+(2s+x)n+2s(s+x). 
\end{align*}
Then we have
\begin{align*}
A&=(n+1)(n+k)\bigl((n+1)^2-(2s+x)(n+1)+2s(s+x)\bigr)\\
&\qquad-n(n-k+1)\bigl(n^2+(2s+x)n+2s(s+x)\bigr)\\
%&=(2n+1)\left\{(k+1)(n+n^2)+k-(k+n+n^2)(2s+x)+2ks(s+x)\right\}\\
&=(2n+1)\left[(k+1-2s-x)(n^2+n)+k\left(1-2s-x+2s(s+x)\right)\right].
\end{align*}
Since
\begin{align*}
n^2+n=(n-k+1)(n+k)+k(k-1)
\end{align*}
and 
\begin{align*}
(k-1)(k+1-2s-x)+1-2s-x+2s(s+x)=L_k,
\end{align*}
we get
\begin{align*}
\mathcal D_s\phi_k
=\left[2(2k+1)(k+1-2s-x)+\frac{2k(2k+1)L_k}{(n-k+1)(n+k)}\right]\phi_k.
\end{align*}
By using \eqref{eq:phi-lower}, we complete the proof.
\qed

For a positive integer $m$, the generalized hypergeometric function $_{m+1}F_m$ is defined by
\begin{align*}
{}_{m+1}F_m
\left[
\begin{matrix}
a_1,\ldots,a_{m+1}\\
b_1,\ldots,b_m
\end{matrix};z
\right]=
\sum_{k=0}^{\infty}
\frac{(a_1)_k\cdots(a_{m+1})_k}
{(b_1)_k\cdots(b_m)_k}
\frac{z^k}{k!}.
\end{align*}
Here, $a_1,\ldots,a_{m+1}$, $b_1,\ldots,b_m$ are complex constants with none of $b_i$ is zero or a negative integer. The series converges absolutely for $|z| < 1$.

Define
\begin{align}
P_s(n)
=(2n+1)
{}_4F_3\!\left[
\begin{matrix}
1-s,\ \frac12-s-x,\ -n,\ n+1\\[1mm]
\frac32,\ 1-\rho,\ 1-\sigma
\end{matrix};1
\right].\label{eq:Ps}
\end{align}
The above ${}_4F_3$ series terminates at $k=s-1$. Hence
\begin{align}
P_s(n)&=(2n+1)\sum_{k=0}^{s-1}\frac{(1-s)_k(\frac12-s-x)_k(-n)_k(n+1)_k}{(\frac32)_k(1-\rho)_k(1-\sigma)_k\,k!}.\label{eq:Psfinite}
\end{align}

It follows immediately from the definition of $P_s(n)$ that
\begin{align}
P_s(0)=1=u_0.\label{eq:P(0)}
\end{align}
For $n=1$, we obtain
\begin{align*}
P_s(1)&=3-\frac{4(1-s)(\frac12-s-x)}{L_1}=\frac{1+x+2s(s+x)}{L_1}.%\label{eq:P1a}
\end{align*}
On the other hand, under the specialization $z=-s(s+x)$,
Lemma~\ref{lem:g-rec} gives
\begin{align*}
g_1=1+x+2s(s+x).
\end{align*}
Since $d_1=L_1$ and $u_1=g_1/d_1$, it follows that
\begin{align}
P_s(1)=\frac{1+x+2s(s+x)}{L_1}=u_1.\label{eq:P(1)}
\end{align}
\begin{rem}
It is worth noting that the hypergeometric function ${}_{4}F_3$ appearing in \eqref{eq:Ps} is an explicit specialization of a Racah polynomial. Indeed, in the standard notation for Racah polynomials \cite{KLS2010}, we have
\begin{align*}
\frac{P_s(n)}{2n+1}=R_{s-1}\left(n(n+1);\frac12,\,-(2s+x)-\rho,\,\rho\right).%\label{eq:Racah-identification}
\end{align*}
\end{rem}

The following lemma shows that $P_s(n)$ provides the solution of the difference equation \eqref{eq:urec} satisfied by $u_n$.
\begin{lem}\label{lem:4F3}
We have 
\begin{align}\label{eq:DP}
\mathcal D_sP_s=-\kappa_sP_s.
\end{align}
Moreover, $u_n=P_s(n)$ for $0\le n\le N$.
\end{lem}
\proof
By \eqref{eq:phi} and \eqref{eq:Psfinite}, we have
\begin{align*}
P_s(n)&=\sum_{k=0}^{s-1}c_k\phi_k(n),%\label{eq:Psfinite}
\end{align*}
where
\begin{align*}
  c_k=\frac{(-1)^k(1-s)_k(\frac12-s-x)_k}{(\frac32)_k(1-\rho)_k(1-\sigma)_k\,k!}.
\end{align*}
For $0\le k\le s-2$, we have
\begin{align*}
\frac{c_{k+1}}{c_k}=-\frac{(k+1-s)(k+\frac12-s-x)}
{(k+1)(k+\frac32)(k+1-\rho)(k+1-\sigma)}=\frac{\lambda_{s-1}-\lambda_k}{\mu_{k+1}}.%\label{eq:ckratio}
\end{align*}
Equivalently,
\begin{align}
c_k\lambda_k+c_{k+1}\mu_{k+1}=c_k\lambda_{s-1}.\label{eq:ratio-match}
\end{align}
Applying $\mathcal D_s$ to $P_s(n)=\sum_{k=0}^{s-1}c_k\phi_k(n)$ and using Lemma \ref{lem:tri-action} , we obtain
\begin{align*}
\mathcal D_sP_s&=\sum\limits_{k=0}^{s-1}c_k\left(\lambda_k\phi_k+\mu_k\phi_{k-1}\right)=\sum\limits_{k=0}^{s-1}c_k\lambda_k\phi_k+\sum\limits_{k=0}^{s-2}c_{k+1}\mu_{k+1}\phi_{k}.%\label{eq:Ps-eigen}
\end{align*}
By \eqref{eq:ratio-match}, we have
\begin{align*}
\mathcal D_sP_s=\lambda_{s-1}\sum_{k=0}^{s-1}c_k\phi_k=\lambda_{s-1}P_s.
\end{align*}
Since $$\lambda_{s-1}=-2(2s-1)(s+x)=-\kappa_s,$$ 
we obtain \eqref{eq:DP}.

Both $u_n$ and $P_s(n)$ satisfy the same recurrence relation and agree at the initial values \eqref{eq:P(0)} and \eqref{eq:P(1)}. Hence, $u_n=P_s(N)$ for $0\le n\le N$. This completes the proof.
\qed

Now we are ready to prove Theorem \ref{thm:g-roots}.
\proofof{Theorem \ref{thm:g-roots}}
Fix $s\in\{1,2,\ldots,N\}$. By Lemma \ref{lem:4F3}, we have
\begin{align*}
u_N=P_s(N)
=(2N+1)
{}_4F_3\!\left[
\begin{matrix}
1-s,\ \frac12-s-x,\ -N,\ N+1\\[1mm]
\frac32,\ 1-\rho,\ 1-\sigma
\end{matrix};1
\right].
\end{align*}
Recall that $2N+1=p$, hence we have $u_N=0$ in $\F_p(x)$. Therefore, $g_N=d_Nu_N=0$. This proves the desired result.
\qed

Finally, we prove Theorem \ref{thm:main-poly} and Theorem \ref{thm:KMY}.
\proofof{Theoren \ref{thm:main-poly}}
Recall that
\begin{align*}
g_N(x,z)=(N!)^2\sum_{r=0}^{N}\sum_{1\le m_1<\cdots<m_r\le N}\prod_{j=1}^{r}\left(\frac{x}{m_j}+\frac{2(-1)^jz}{m_j^2}\right).
\end{align*}
By Lemma \ref{lem:g-rec}, $g_N(x,z)$ is a monic polynomial of degree $N$ in $z$. It is clear that $-s(s+x)(1\le s\le N)$ are pairwise distinct in $\F_p(x)$. Indeed, for $v\neq s$,
\begin{align*}
-s(s+x)+v(v+x)=(v-s)(v+s+x)\neq 0.
\end{align*}  
By Theorem \ref{thm:g-roots}, $g_N(x,z)$ has $N$ distinct roots. Therefore, 
\begin{align*}
g_N(x,z)=\prod\limits_{s=1}^{N}(z+s(s+x)).
\end{align*}
Dividing both sides by $(N!)^2$ yields \eqref{eq:main-poly}.
\qed

\proofof{Theorem \ref{thm:KMY}}
The coefficient of $x^{r-a}z^a$ in the left-hand side of \eqref{eq:main-poly} is 
\begin{align*}
2^a\sum_{\substack{k_i\in\{1,2\}\\
\card\{i:k_i=2\}=a}}(-1)^{\card\{i:\ i\text{ odd},\ k_i=2\}}\sum_{1\le m_1<\cdots<m_r\le N}\frac{1}{m_1^{k_1}\cdots m_r^{k_r}}.%\label{eq:coefF}
\end{align*}
On the other hand, the coefficient of $x^{r-a}z^a$ in the right-hand side of \eqref{eq:main-poly} is 
\begin{align*}
\sum_{\substack{k_i\in\{1,2\}\\\card\{i:k_i=2\}=a}}\sum_{1\le m_1<\cdots<m_r\le N}\frac{1}{m_1^{k_1}\cdots m_r^{k_r}}.%\label{eq:coefG}
\end{align*}
Since $N=(p-1)/2$, comparing the coefficients of $x^{r-a}z^a$ on both sides of \eqref{eq:main-poly} yields the desired result.
\qed

\section*{Acknowledgments}
The first author is supported by the Natural Science Foundation of Shanghai (Grant No. 24ZR1469000).

\end{document}